\documentclass[preprint,12pt]{elsarticle}

\usepackage{amssymb,amsmath}
\usepackage{amsthm}

\newtheorem{theorem}{Theorem}[section]
\newtheorem{lemma}[theorem]{Lemma}

\numberwithin{equation}{section}

\newcommand{\Der}{\mathrm{Der}}

\begin{document}

\begin{frontmatter}



\title{A Note on Darboux Polynomials of Monomial Derivations\tnoteref{label1}}

\tnotetext[label1]{Supported by
the Youth Research Funds from
Liaoning University under Grant No. LDQN201428.}

\author{Jiantao Li}
\ead{jtlimath@gmail.com}
\address{School of Mathematics, Liaoning University, Shenyang, 110031, China}
\begin{abstract}
We study a monomial derivation $d$ proposed by J. Moulin Ollagnier and A. Nowicki in the polynomial ring of four variables, and prove that $d$ has no Darboux polynomials if and only if $d$ has a trivial field of constants.
\end{abstract}


\end{frontmatter}


\section{Introduction}

Throughout this paper, let $k[X]=k[x_1, x_2, \ldots, x_n]$ denote
the polynomial ring over a field $k$ of characteristic 0.

 A derivation
$d=f_1\frac{\partial}{\partial
x_1}+\cdots+f_n\frac{\partial}{\partial x_n}$ of $k[X]$ is said to
be a monomial derivation if each  $f_i$ is a monomial in $k[X]$. By
a Darboux polynomial of $d$ we mean a polynomial $F\in k[X]$ such
that $F\notin k$ and $d(F)=\Lambda F$ for some $\Lambda \in k[X]$.

Derivations and Darboux polynomials are useful algebraic methods to
study polynomial or rational differential systems. If we associate a
polynomial differential system $\frac{d}{dt}x_i=f_i,\ i=1,\ldots,n,$
with a derivation $d=f_1\frac{\partial}{\partial
x_1}+\cdots+f_n\frac{\partial}{\partial x_n}$, then the existence of
Darboux polynomials for $d$ is a necessary condition for the system to have
a first integral (see \cite{Mac1,Mac2,Mou}). Darboux polynomials also have important applications in
many branches of mathematics.  The famous Jacobian conjecture for $k[X]$ is
equivalent to the assertion that $\frac{\partial}{\partial
x_1},\ldots,\frac{\partial}{\partial
x_n}$ is,
apart from a polynomial coordinate change, the only commutative
$k[X]$-basis of $\Der_k~k[X]$. It is proved that $n$ pairwise commuting derivations form a commutative basis if and only if they are $k$-linearly independent and have no common Darboux polynomials \cite{Li}. 

The most famous derivation without Darboux polynomials may be the Jouanolou derivation, there are several different proofs on the fact that Jouanolou derivations have no Darboux polynomials, see\cite{Mou2,Zoladek}. It is obvious that if $d$ is without Darboux polynomials, then the field $k(X)^d$ is trivial. The opposite implication is, in general, not true. In \cite{NowickiZ06}, there is a full description of all monomial derivations of $k[x,y,z]$ with trivial field of constants. Using this description and several additional facts, Moulin-Ollagnier and Nowicki present full lists of homogeneous monomial derivations of degrees $s\leq4$ (of $k[x,y,z]$) without Darboux polynomials in \cite{MoulinN08} and then in \cite{MoulinN11}, they prove that a monomial derivation $d$ (of $k[x,y,z]$) has no Darboux polynomials if and only if $d$ has a trivial field of constants and $x_i \nmid d(x_i)$ for all $i=1,\dots,n$.

More precisely, look at a monomial derivation $d$ of $k[X]$ with $d(x_i)= x_1^{\beta_{i1}}\cdots x_n^{\beta_{in}}$
for $i=1,\dots,n$ and  each $\beta_{ij}$ is a non-negative integer. In this case, $d$ is said to be normal monomial if $\beta_{11}=\beta_{22}=\cdots=\beta_{nn} = 0$ and $w_d\neq0$, where $w_d$ is the determinant of the matrix $[\beta_{ij}]-I$. In \cite{MoulinN11}, it is proved that if $d$ is a normal monomial derivation of $k[X]$, then $d$ is without Darboux polynomials if and only if $k(X)^d = k$. What happens if $w_d=0$? In \cite{MoulinN11}, a monomial derivation $d$  of $k[x,y,z,t]$  with  $w_d=0$ defined by $$d(x)=t^2, d(y)=zt, d(z)=y^2, d(t)=xy$$ is proposed. In this note, we prove that $d$ has no Darboux polynomial if and only if $d$ has a trivial field of constant.

\section{Main Results}

Now we recall some lemmas related to Darboux polynomials of polynomial derivations. Denote by $A_\gamma^{(s)}$  the group of all  $\gamma$-homogeneous polynomials of degree $s$ in $k[X]$. Then $k[X]$ becomes a $\gamma$-graded ring $k[X]=\oplus_{s\in\mathbb{Z}}A_\gamma^{(s)}$. Recall that $D$ is said to be a
$\gamma$-homogeneous derivation of degree $s$ if $D(A_\gamma^{(p)})\subseteq A_\gamma^{(s+p)}$ for any $p\in \mathbb{Z}$.
\begin{lemma}\textup{\cite[Proposition 2.2.1]{Nowicki1994}}
Let~$f$ be a Darboux polynomial of $D$. Then all factors of $f$ are also Darboux polynomials of $D$.
\end{lemma}

\begin{lemma}\textup{\cite[Proposition 2.2.3]{Nowicki1994}}\label{C3Darboux homo}
Let $D$ be a $\gamma$-homogeneous derivation of degree $s$ and  $f$  be a Darboux polynomial of $D$ and $\lambda$ be a polynomial eigenvalue of $f$ with respect to $D$. Then $\lambda$ is a $\gamma$-homogeneous polynomial of degree $s$, and every $\gamma$-homogeneous component of $f$ is also a Darboux polynomial of $D$ with polynomial eigenvalue  $\lambda$.
\end{lemma}

Now consider a monomial derivation $d$ defined by
$d(x_i)=X^{\beta_i}$, where $\beta_i=(\beta_{i1}, \ldots,
\beta_{in})\in \mathbb{N}^n$. Write $\beta=[\beta_{ij}]$,
$\alpha=[\alpha_{ij}]=\beta-I$, where $I$ is the identity matrix
of order $n$. Let $w_d=\det \alpha$, that is,
$$
w_D=\det \alpha=
 \begin{vmatrix}
  \beta_{11}-1 & \beta_{12} & \dots & \beta_{1n}  \\
  \beta_{21} & \beta_{22}-1 & \dots & \beta_{2n}  \\
  \vdots & \vdots &  & \vdots  \\
  \beta_{n1} & \beta_{n2} & \dots & \beta_{nn}-1  \\
 \end{vmatrix}.
$$

Look at the monomial derivation $d $ of $k[x, y, z, t]$ defined by
$$d(x)=t^2, d(y)=zt, d(z)=y^2, d(t)=xy.$$

\begin{theorem}
$d$ has no Darboux
polynomials if and only if $d$ has a trivial field of constants.
\end{theorem}

\begin{proof}
It is obvious that if $d$ is without Darboux polynomials, then the field $k(X)^d$ is trivial.

Now suppose that $k(X)^d$ is trivial.
Assume that $d$ has a Darboux $F$ such that $d(F)=\Lambda F$. Since $d$ is a homogeneous derivation of degree $1$, then by Lemma~\ref{C3Darboux homo},
we have $\Lambda$ is a homogeneous polynomial of degree $1$, thus $\Lambda=k_1x+k_2y+k_3z+k_4t, k_1,\dots,k_4\in k$.

Let $\sigma: k[x,y,z,t]\rightarrow k[x,y,z,t]$ be an automorphism defined by:
$$\sigma(x)=\varepsilon^3 x,~~\sigma(y)=\varepsilon^5y,~~\sigma(z)=\varepsilon^3z, \sigma(t)=\varepsilon t,$$
where  $\varepsilon$   is a primitive eighth root of $1$.
Then $\sigma^{-1}$ is:
$$\sigma^{-1}(x)=\varepsilon^5 x,~~\sigma^{-1}(y)=\varepsilon^3y,~~\sigma(z)^{-1}=\varepsilon^5z, \sigma^{-1}(t)=\varepsilon^7 t.$$

It is easy to verify that $$\sigma^{-1}d\sigma(x)=\sigma^{-1}d(\varepsilon^3x)=\sigma^{-1}(\varepsilon^3t^2)=\varepsilon^{17} t^2=\varepsilon t^2,$$
$$\sigma^{-1}d\sigma(y)=\sigma^{-1}d(\varepsilon^5y)=\sigma^{-1}(\varepsilon^5zt)=\varepsilon^{17} zt=\varepsilon zt,$$
$$\sigma^{-1}d\sigma(z)=\sigma^{-1}d(\varepsilon^3z)=\sigma^{-1}(\varepsilon^3y^2)=\varepsilon^9 y^2=\varepsilon y^2,$$
$$\sigma^{-1}d\sigma(t)=\sigma^{-1}d(\varepsilon t)=\sigma^{-1}(\varepsilon xy)=\varepsilon^9 xy=\varepsilon xy.$$
Thus,
$$\sigma^{-1}d\sigma=\varepsilon d, \text{moreover,} ~~\sigma^{-i}d\sigma^i=\varepsilon^i d.$$

Let
$$\bar{F}=\prod_{i=0}^{7}\sigma^i(F),~~\bar{\Lambda}=\sum_{i=0}^{7}\varepsilon^i\sigma^i(\Lambda).$$

Then
\begin{equation*}
\begin{split}
d(\bar{F})&=d(\prod_{i=0}^{7}\sigma^i(F))=\sum_{i=0}^{7}\sigma^0(F)\cdots d(\sigma^i(F))\cdots \sigma^{7}(F)\\
&=\sum_{i=0}^{7}\sigma^0(F)\cdots \varepsilon^i\sigma^i(d(F))\cdots \sigma^{7}(F)\\
&=\sum_{i=0}^{7}\sigma^0(F)\cdots \varepsilon^i\sigma^i(\Lambda F)\cdots \sigma^7(F)\\
&=\sum_{i=0}^{7}\sigma^0(F)\cdots \varepsilon^i\sigma^i(\Lambda) \sigma^i(F)\cdots \sigma^7(F)\\
&=(\sum_{i=0}^7\varepsilon^i\sigma^i(\Lambda))\prod_{i=0}^7\sigma^i(F)=\bar{\Lambda} \bar{F}.
\end{split}
\end{equation*}

Thus, $\bar{F}$ is a Darboux polynomial of $d$ with  eigenvalue $\bar{\Lambda}$.
Since $\varepsilon$   is a primitive eighth root of $1$, we have
$$\sum_{i=0}^{7}\varepsilon^{ri}=\frac{1-\varepsilon^{8r}}{1-\varepsilon^{r}}=0, \text{for any}~~ r<8.$$
Thus£¬
\begin{equation*}
\begin{split}
\bar{\Lambda}&= \sum_{i=0}^7\varepsilon^i\sigma^i(\Lambda)\\
&=\sum_{i=0}^7\varepsilon^i\sigma^i(k_1x+k_2y+k_3z+k_4t)\\
&=\sum_{i=0}^{7}\varepsilon^i(k_1\varepsilon^{3i}x+k_2\varepsilon^{5i}y+k_3\varepsilon^{3i}z+k_4\varepsilon^it)\\
&=k_1\sum_{i=0}^{7}\varepsilon^{4i}x+k_2\sum_{i=0}^{7}\varepsilon^{6i}y+k_3\sum_{i=0}^{7}\varepsilon^{4i}z+k_4\sum_{i=0}^{7}\varepsilon^{2i}t\\
&=0.
\end{split}
\end{equation*}
Therefore, $D(\bar{F})=0$. It is a contradiction to the fact that $k(X)^d$. Hence, $d$ has no Darboux polynomials.
\end{proof}

\bibliographystyle{amsplain}

\end{document}